\documentclass{amsart}
\usepackage[english,frenchb]{babel}
\usepackage{lmodern}
\usepackage{enumerate}
\usepackage{amsmath}
\usepackage{amsfonts}
\usepackage{amssymb}
\usepackage{amsthm}
\usepackage{graphicx}
\usepackage{mathrsfs}
\usepackage{pxfonts}
\usepackage{datetime}
\usepackage{comment}
\usepackage{fourier-orns}
\usepackage{dsfont}
\usepackage[left=3cm,right=3cm,top=3cm,bottom=3cm]{geometry}
\usepackage[cyr]{aeguill}
\usepackage{fancyhdr}
\usepackage{mathabx}
\usepackage[all,cmtip]{xy}
\usepackage[latin1]{inputenc}
\usepackage{amscd}
\usepackage{enumitem}
\usepackage{hyperref}
\usepackage{amsmath,amscd}
\usepackage{tikz-cd}
\usetikzlibrary{decorations.pathmorphing}

\usepackage{amsfonts,amsmath,graphicx}
\newcommand{\bc}{\begin{center}}
\newcommand{\ec}{\end{center}}
\newcommand{\bq}{\begin{quote}}
\newcommand{\eq}{\end{quote}}
\newcommand{\bqtn}{\begin{quotation}}
\newcommand{\eqtn}{\end{quotation}}
\newcommand{\beq}{\begin{equation}}
\newcommand{\eeq}{\end{equation}}
\newcommand{\bearr}{\begin{eqnarray}}
\newcommand{\eearr}{\end{eqnarray}}
\newcommand{\bearrn}{\begin{eqnarray*}}
\newcommand{\eearrn}{\end{eqnarray*}}
\newcommand{\bi}{\begin{itemize}}
\newcommand{\ei}{\end{itemize}}
\newcommand{\be}{\begin{enumerate}}
\newcommand{\ee}{\end{enumerate}}
\newcommand{\bthe}{\begin{theorem}}
\newcommand{\ethe}{\end{theorem}}
\newcommand{\blem}{\begin{lemme}}
\newcommand{\elem}{\end{lemme}}
\newcommand{\bsolu}{\begin{solution}}
\newcommand{\esolu}{\end{solution}}
\newcommand{\bexer}{\begin{exercise}}
\newcommand{\eexer}{\end{exercise}}

\newcommand{\mcA}{\mathcal{A}}

\newcommand{\CC}{\mathbb{C}}
\newcommand{\RR}{\mathbb{R}}

\newcommand{\ba}{\begin{array}}
\newcommand{\ea}{\end{array}}

\usepackage{amsfonts,amsmath}
\newtheorem{theoreme}{Theorem}[section]
\newtheorem{theorem}[theoreme]{Theorem}
\newtheorem{lemme}[theoreme]{Lemma}
\newtheorem{lemma}[theoreme]{Lemma}
\newtheorem{proposition}[theoreme]{Proposition}
\newtheorem{definition}[theoreme]{Definition}

\newtheorem{corollaire}[theoreme]{Corollary}
\newtheorem{corollary}[theoreme]{Corollary}

\newtheorem{solution}[theoreme]{Solution}
\newtheorem{exercise}[theoreme]{Exercise}

\newcommand{\bdefi}{\begin{definition}}
\newcommand{\edefi}{\end{definition}}
\newcommand{\brk}{\begin{remarque}}
\newcommand{\erk}{\end{remarque}}
\newcommand{\bpp}{\begin{proposition}}
\newcommand{\epp}{\end{proposition}}
\newcommand{\bpf}{\begin{proof}}
\newcommand{\epf}{\end{proof}}
\newcommand{\bcor}{\begin{corollaire}}
\newcommand{\ecor}{\end{corollaire}}
\newcommand{\bsol}{\begin{solution}}
\newcommand{\esol}{\end{solution}}
\theoremstyle{definition}

\newtheorem{remarque}[theoreme]{Remark}

\allowdisplaybreaks
\title{A rotation group whose subspace arrangement is not from a real reflection group}
\author{Ben Blum-Smith}
\email{blumsmib@newschool.edu}

\begin{document}
\large
\selectlanguage{english}
\maketitle

\begin{abstract}
We exhibit a family of real rotation groups whose subspace arrangements are not contained in that of any real reflection group, answering a question of Martino and Singh.

\smallskip
\noindent \textsc{Keywords.} Rotation groups; reflection groups; subspace arrangements; coxeter groups

\smallskip
\noindent \textsc{MSC Classes.} 14N20; 20F55; 51F15; 05E15
\end{abstract}

\section{Introduction}

Let $V$ be a finite-dimensional vector space over some field, and let $G$ be a finite subgroup of the general linear group $GL(V)$ on $V$. An element $g\in G$ is a {\em reflection} if its 1-eigenspace has codimension one, and a {\em bireflection} if it has codimension at most two. If $g$'s 1-eigenspace has codimension exactly two and $V$ is an $\RR$-vector space, $g$ is also called a {\em rotation}. If $G$ is generated by its reflections (respectively bireflections, rotations), it is called a {\em reflection} (respectively {\em bireflection}, {\em rotation}) {\em group}.

Reflection groups are a topic of longstanding interest, with a beautiful and developed body of theory (e.g. see \cite{Hum90}, \cite{LT09} for the real and complex cases, respectively). As an illustration, they are precisely those finite linear groups whose algebras of invariants are polynomial, at least when the field characteristic does not divide the group order, by the celebrated Chevalley-Shephard-Todd theorem. 

In the last four decades, bireflection groups have also turned out to have invariant-theoretic significance; see \cite{ES80}, \cite{KW82}, \cite{Kem99}, \cite{LP01}, \cite{GK03}, \cite{Duf09}, \cite{DEK09}, \cite{BSM18}. In addition, in the real case, Christian Lange \cite{Lan16} has recently shown that bireflection groups (also known, when working over $\RR$, as {\em reflection-rotation groups}) are characterized by a remarkable topological property: they are exactly those finite groups $G$ of linear automorphisms of $\RR^n$ such that the orbifold $\RR^n/G$ is a piecewise-linear manifold (possibly with boundary). Lange's result is by way of a full classification of reflection-rotation groups jointly accomplished with Marina Mikha\^{i}lova \cite{LM16}, completing a program developed by Mikha\^{i}lova in the 1970's and 1980's (\cite{Mae76}, \cite{Mik78}, \cite{Mik82}, \cite{Mik85}).

Thus bireflection groups have emerged as a natural generalization of reflection groups and a natural object of study in their own right. It is in this context that Martino and Singh \cite{MS19} studied the subspace arrangements of such groups. For a finite subgroup $G\subset GL(V)$, the {\em subspace arrangement $\mcA^G$ of $G$} is the arrangement of subspaces (in $V$) of fixed points of nontrivial isotropy subgroups of $G$. 
For finite (real or complex) reflection groups $W$, the subspace arrangement $\mcA^W$ is just the set of reflecting hyperplanes and their intersections. This is a well-studied object (e.g. \cite{OS82}, \cite{OS83}, and Chapter 6 of \cite{OT92}). Working in the real case, Martino and Singh verified that for any rotation group $G$ of degree $\leq 3$ (i.e. for which $\dim V\leq 3$), there exists a (real) reflection group $W$ acting on $V$ such that $\mcA^G$ is contained in $\mcA^W$, and they asked (\cite[Question 1]{MS19}) whether this holds in general.

The purpose of this note is to present an example showing that it does not, and in fact, a failure already occurs for an infinite sequence of groups in degree 4.

\begin{theorem}\label{thm:main}
Let $G = G(m,1,2)\subset GL(2,\CC)$. Consider $G$ as a subgroup of the orthogonal group $O_4(\RR)$ by identifying $\CC^2$ with $\RR^4$. Then for $m$ sufficiently large, $G$ is a rotation group such that $\mcA^G$ is not contained in $\mcA^W$ for any real reflection group $W\subset O_4(\RR)$.
\end{theorem}

Notation and terminology are fixed in section \ref{sec:prelims}, where we also recall the needed basic facts from the theory of reflection groups. The proof is in section \ref{sec:example}.

We expect that a closer study of Lange and Mikha\^{i}lova's classification \cite{LM16} will reveal that, apart from those rotation groups that are explicitly constructed as subgroups of reflection groups (e.g. the alternating subgroups of reflection groups), almost all rotation groups will have subspace arrangements not contained in that of any reflection group.

\section{Acknowledgement}

The author would like to express gratitude to Victor Reiner who directed him to Martino and Singh's paper \cite{MS19}, and to Ivan Martino for the fruitful discussion that led to this note. Additionally, he would like to acknowledge Yves de Cornulier who in \cite{Cor14} gave the above group $G$ (in the special case that $m$ is prime) as an example of a finite subgroup of $O_n(\RR)$ that is not contained in a reflection group. It was in view of this that $G$ recommended itself to the author as a possible answer to Martino and Singh's question.

\section{Preliminaries}\label{sec:prelims}

By and large we follow the notation of \cite{MS19}. Fix a vector space $V$ and a finite subgroup $G$ of the general linear group $GL(V)$ on $V$. For any subgroup $H\subset G$, $V^H$ denotes the subspace of fixed points of $H$:
\[
V^H = \{v\in V: hv = v, \forall h\in H\}.
\]
Recall that an {\em isotropy subgroup} $H$ for the action of $G$ on $V$ is the stabilizer 
\[
\{g\in G: gv = v\}
\]
of a vector $v\in V$. 
As mentioned above, the {\em subspace arrangement of $G$} is the collection of the subspaces $V^H$ as $H$ varies over the nontrivial isotropy subgroups of $G$:
\[
\mcA^G = \{ V^H : H\neq \{e\}\text{ is an isotropy subgroup}\}.
\]
If $W$ is a (real or complex) reflection group, then $\mcA^W$ consists precisely of the reflecting hyperplanes of $W$ and their intersections \cite[Theorem 6.27]{OT92}. (The {\em reflecting hyperplanes} are the 1-eigenspaces of the reflections in $W$.) This is a consequence of the fact (\cite[\S 1.12]{Hum90}, \cite[Theorem 9.44]{LT09}) that the isotropy subgroups in a reflection group are generated by the reflections they contain.

Throughout, for any group of linear automorphisms of a vector space, we use the word {\em degree [of the group]} to refer to the dimension of the space being acted on.

For a triple of natural numbers $(m,p,n)$ with $p$ a divisor of $m$, the group $G(m,p,n)\subset U_n$ is defined to be the group of unitary $n\times n$ matrices generated by (a) the diagonal matrices whose diagonal entries are $m$th roots of unity and whose determinant is an $m/p$th root of unity, and (b) the permutation matrices. This notation was introduced by Shephard and Todd in \cite{ST54}, which gives a complete classification of complex reflection groups. (For an exposition of this classification, see \cite[Chapter 8]{LT09}.)

The group of interest in this note is $G = G(m,1,2)$, consisting of the $2\times 2$ matrices of the form
\[
\begin{pmatrix}
\zeta^j & \\
 & \zeta^k
\end{pmatrix},\;
\begin{pmatrix}
 &\zeta^k\\
\zeta^j& 
\end{pmatrix}
\]
where $\zeta = e^{2\pi i / m}$ and $j,k$ are any integers modulo $m$. Abstractly, this is the wreath product $C_m\wr S_2$ of a cyclic group of order $m$ by the symmetric group on 2 letters.

Any complex reflection group, and thus $G$ in particular, becomes a real rotation group by regarding the space on which it acts as a real vector space. Thus, identify $\CC$ with $\RR^2$ via the $\RR$-basis $1,i$. Then multiplication by $e^{\theta i}\in \CC$ (for $\theta\in [0,2\pi)$) acts on $\RR^2$ via the standard $2\times2$ rotation matrix 
\[
R_\theta = \begin{pmatrix} \cos\theta&-\sin\theta\\ \sin\theta& \cos\theta\end{pmatrix},
\]
so the canonical action of $G$ on $\CC^2$ by matrix multiplication becomes an action on $\RR^4$, embedding $G$ in $O_4(\RR)$ via
\[
\begin{pmatrix}
\zeta^j & \\
 &\zeta^k
\end{pmatrix}
\mapsto
\begin{pmatrix}
R_{2\pi j/m} & \\
 & R_{2\pi k/m}
\end{pmatrix},\;
\begin{pmatrix}
 & \zeta^k\\
\zeta^j & 
\end{pmatrix}
\mapsto
\begin{pmatrix}
 & R_{2\pi k/m}\\
R_{2\pi j/m}
\end{pmatrix}
\]
where the images are $4\times 4$ matrices in block form. Let $x,y$ be the standard coordinate functions on $\CC^2$, which become complex-valued functions on $\RR^4$ via the identification. For the rest of this note, we fix $V=\CC^2\simeq \RR^4$ and move freely between these real and complex points of view. We reserve the words {\em line} and {\em plane} for when we are working in the real point of view: they always mean linear subspaces of $V$ of $\RR$-dimension one and two respectively.

We make use of the classification theorem for (finite) real reflection groups and its standard notation (see \cite[Chapter 2]{Hum90} for example). In particular, we need the facts that (i) a reflection group splits as a direct product of irreducible factors acting in orthogonal subspaces (see \cite[\S 2.2]{Hum90}), and that (ii) the only degree in which there are infinitely many inequivalent irreducible real reflection groups is 2 (see \cite[\S 2.4-2.7]{Hum90}).

We mention that what we here call bireflection groups, following \cite{Kem99}, \cite{LP01}, \cite{GK03}, \cite{DEK09}, are referred to by Martino and Singh in \cite{MS19} as {\em groups generated in codimension two}. If a group is generated by bireflections and contains no reflections, Martino and Singh say it is {\em strictly} generated in condimension two. The terminology of reflection-rotation groups, i.e. bireflection groups over $\RR$, is used by \cite{LM16}, \cite{Lan16}, and \cite{MS19} as well.

\section{The example}\label{sec:example}

Take $G=G(m,1,2)$ as above. See the previous section for all notation; in particular recall that $x,y$ are the (complex-valued) coordinate functions on $V$ seen as $\CC^2$.

We begin by describing $\mcA^G$ explicitly. This description follows from the more general description of $\mcA^G$ for $G(m,1,n)$ given in \cite[Example 6.29 and \S 6.4]{OT92}. (One verifies it easily by adopting the complex point of view and looking for the 1-eigenspaces of the matrices in $G$.)

\begin{lemma}\label{lem:AG}
The nonzero elements in the subspace arrangement $\mcA^G$ of $G\subset O_4(\RR)$ consist of $m+2$ planes that intersect pairwise trivially, and nothing else. They can be described by the equations $x=0$, $y=0$, and $y=\zeta^j x$ for $j=0,\dots,m-1$.\qed
\end{lemma}



Now we view $V$ as $\RR^4$ and show that the elements of $\mcA^G$ just described cannot all nontrivially meet the same plane. The plan of the proof of Theorem \ref{thm:main}, in outline, is to show that for $m$ sufficiently big, all reflection groups $W$ of degree four such that $\mcA^W$ is big enough to contain $\mcA^G$ have the property that there is a pair of planes $V_1,V_2$ in $V$ such that all of the planes in $\mcA^W$ are either equal to one of them or meet both (nontrivially). It will follow that the planes of $\mcA^G$ cannot be among these, since they do not all meet the same plane and also do not meet each other.

\begin{lemma}\label{lem:plane}
Suppose $P\subset V$ is a plane (i.e. {\em real} two-dimensional subspace) that meets both $\{x=0\}$ and $\{y=0\}$ nontrivially. Then it meets at most one of the other subspaces $\{y=\zeta^j x\}$ nontrivially.
\end{lemma}

\begin{proof}
Let $v\in P\cap \{y=0\}$ and $w\in P\cap \{x=0\}$ be nontrivial points of intersection. Let 
\[
x^\star=x(v) 
\]
and 
\[
y^\star=y(w).
\]
Note that
\[
y(v) = x(w) = 0,
\] 
while $x^\star,y^\star$ are nonzero since $v,w$ are nontrivial. Since $\{x=0\}$ and $\{y=0\}$ are complementary, $v,w$ are linearly independent, and thus we have
\[
P = \{av + bw : a,b\in\RR\}.
\]
Since $x,y$ are $\RR$-linear (in fact, $\CC$-linear, but we don't need this) functions $V\rightarrow \CC$, we have
\begin{align*}
x(av + bw) &= a x^\star\\
y(av + bw) &= b y^\star.
\end{align*}
If $u =  av + bw\in P$ is any point that does not lie on $\{x=0\}$ or $\{y=0\}$, it follows that $a,b\neq 0$. Since $a,b$ are real,  $b/a$ is real, and it further follows that the nonzero complex number $y(u) / x(u) = (b/a)(y^\star / x^\star)$ has the same argument as $y^\star / x^\star$, i.e. its argument does not depend on $a$ or $b$. Thus the function $\varphi$ from $P' = P\setminus (\{x=0\} \cup \{y=0\})$ to the unit circle $S^1 = \{z\in \CC: |z|=1\}\subset \CC$ given by
\[
\varphi(u) = \frac{y(u)/x(u)}{|y(u)/x(u)|}
\]
is defined and has a constant value on all of $P'$.

Now if $P$ intersects any $\{y=\zeta^jx\}$ nontrivially, it does so away from $\{x=0\}$ and $\{y=0\}$, since $\{x=0\}\cap \{y=\zeta^jx\} = \{y=0\}\cap \{y=\zeta^j x\} = \{0\}$ by Lemma \ref{lem:AG}. It follows that $\varphi$ is defined and constant along the set of all nonzero intersection points between $P$ and any of the subspaces $\{y=\zeta^j x\}$. In particular $\varphi$ can take at most one of the values $\zeta^j$, $j=0,\dots,m-1$, along this set of intersection points.

But meanwhile, $\varphi$ has a constant value of $\zeta^j$ along each of the subspaces $\{y=\zeta^jx\}$ (away from zero). It follows that $P$ can intersect at most one of them nontrivially.
\end{proof}

\begin{corollary}\label{cor:plane}
If $m\geq 2$, no plane $P\subset V$ can intersect every one of the subspaces described by Lemma \ref{lem:AG} nontrivially.\qed
\end{corollary}

\begin{proof}[Proof of Theorem \ref{thm:main}]
For the remainder of the argument, we ensconce ourselves in the real point of view. We show that the subspace arrangement $\mcA^G$ described by Lemma \ref{lem:AG} cannot be contained in the arrangement $\mcA^W$ for any real reflection group $W$, as long as $m$ is sufficiently high.

Now $\mcA^G$ cannot be contained in any $\mcA^W$ that it exceeds in cardinality. There are only a finite number of inequivalent irreducible real reflection groups $W$ in any degree except two, and therefore (in view of the fact that all reflection groups split as direct products of irreducible factors) only a finite number of real reflection groups in degree four containing an irreducible direct factor of degree three or four. Thus it is possible to choose $m$ large enough so that $|\mcA^G|$ is larger than $|\mcA^W|$ for any $W$ chosen from among these, since $|\mcA^G|$ rises linearly with $m$ by Lemma \ref{lem:AG}.

We have therefore achieved the stated result unless $W$ is reducible with all irreducible direct factors of degree $\leq 2$, so we can restrict attention to this case. Then $V$ has an orthogonal decomposition $V_1\oplus V_2$ with $\dim_\RR V_1=\dim_\RR V_2 = 2$, such that $W$ is generated by elements that act as reflections in $V_1$ and fix $V_2$ pointwise, and elements that act as reflections in $V_2$ and fix $V_1$ pointwise. (We do not assume here that $W$'s actions in $V_1$ and $V_2$ are irreducible.) Thus all the reflecting hyperplanes of $W$ either contain $V_2$ and intersect $V_1$ in a line, or vice versa.

Since $\mcA^W$ consists of the reflecting hyperplanes and their intersections, the planes of $\mcA^W$ are precisely the pairwise intersections between reflecting hyperplanes. These can have two forms:

Any two distinct reflecting hyperplanes containing $V_1$ have a two-dimensional intersection containing, and thus equaling, $V_1$. So it is possible that $V_1\in \mcA^W$. Likewise, the possibility that there are two reflecting hyperplanes both containing $V_2$ shows $V_2$ can be in $\mcA^W$.

If a plane of $\mcA^W$ is not $V_1$ or $V_2$, it is not the intersection of two reflecting hyperplanes that both contain one of them. Thus it must be the intersection of a reflecting hyperplane $M_1$ that contains $V_1$ and intersects $V_2$ in a line, with another reflecting hyperplane $M_2$ that contains $V_2$ and intersects $V_1$ in a line. This intersection $M_1\cap M_2$ is a plane containing, and thus spanned by, the lines $M_1\cap V_2$ and $M_2\cap V_1$. (Note they are distinct because $V_1\perp V_2$.) In particular, it meets both $V_1$ and $V_2$ in a line. Thus {\em every element of $\mcA^W$ that is not $V_1$ or $V_2$ must meet each of them in a line.}

To summarize: if $W$ is a real reflection group with no irreducible factors of degree greater than $2$, then there exists a pair of orthogonal planes $V_1,V_2\subset V$ such that every plane in $\mcA^W$ has the property that either it is equal to one of them, or it meets both of them in a line. Thus, if $\mcA^G\subset\mcA^W$, then each of the planes described by Lemma \ref{lem:AG} has this property as well.

For $m\geq 2$, this is impossible. It cannot be that any of the planes of Lemma \ref{lem:AG} is equal to $V_1$ or $V_2$, for all but at most one of the others would then have to meet it in a line, and yet we know all the pairwise intersections are trivial. Meanwhile, if none of them is equal to $V_1$ or $V_2$, then {\em all} of them have to meet both $V_1$ and $V_2$ nontrivially. But Corollary \ref{cor:plane} shows that if $m\geq 2$, it is not possible for all of them to meet either $V_1$ or $V_2$ nontrivially, let alone both.
\end{proof}

\end{document}